\documentclass[11pt]{article}
\usepackage[utf8]{inputenc}
\usepackage{graphicx}
\usepackage{latexsym}
\usepackage{amsmath,bm}
\usepackage{amssymb}
\usepackage{amsthm}
\usepackage{hyperref}
\usepackage{fullpage}
\usepackage{multicol}
\usepackage{subcaption}
\usepackage{multirow}
\usepackage{caption}
\usepackage{MnSymbol}
\usepackage{mathtools}
\usepackage{array}
\usepackage{lipsum}
\usepackage{caption,multicol}
\usepackage{subcaption}
\def\@fnsymbol#1{\ensuremath{\ifcase#1\or *\or \dagger\or \ddagger\or
		\mathsection\or \mathparagraph\or \|\or **\or \dagger\dagger
		\or \ddagger\ddagger \else\@ctrerr\fi}}
\newcommand{\ssymbol}[1]{^{\@fnsymbol{#1}}}

\DeclareMathOperator{\cir}{Circ}
\DeclareMathOperator{\In}{In}
\DeclareMathOperator{\ra}{rank}
\DeclareMathOperator{\s}{span}
\DeclareMathOperator{\diag}{diag}
\usepackage[most]{tcolorbox}
\allowdisplaybreaks

\newtheorem{theorem}{Theorem}[section]
\newtheorem{remark}[theorem]{Remark}

\newtheorem{lemma}[theorem]{Lemma}

\newtheorem{definition}[theorem]{Definition}

\numberwithin{equation}{section}
\title{The Moore-Penrose Inverse of the Distance Matrix  of a Helm Graph}
  
 \author{I. Jeyaraman\thanks{jeyaraman@nitt.edu},
 	T. Divyadevi\thanks{tdivyadevi@gmail.com} 
 	and  R. Azhagendran\thanks{razhagu2014@gmail.com}
 	\\\\
 	Department of Mathematics\\
 	National Institute of 
 	Technology		Tiruchirappalli-620 015,  India 
 	 }
 
\date{\today}
\begin{document}
	\maketitle
\begin{abstract}
		
		In this paper, we give necessary and sufficient conditions for
		a  real symmetric matrix, and in particular, for the distance matrix
		$D(H_n)$ of a helm graph
		$H_n$ to have
		their Moore-Penrose inverses as the sum of a symmetric Laplacian-like
		matrix and a rank one matrix. As a
		consequence,
		we present a short proof of the inverse formula, given by 	Goel
		(Linear Algebra Appl. 621:86--104, 2021), for $D(H_n)$
		when $n$ is even. Further, we
		derive
		a  formula for the Moore-Penrose inverse of singular
		$D(H_n)$ that is analogous to the formula for $D(H_n)^{-1}$.
		Precisely, if $n$ is odd, we find a symmetric positive semidefinite
		Laplacian-like
		matrix
		$L$ of order
		$2n-1$ and a vector $\mathbf{w}\in \mathbb{R}^{2n-1}$ such that
		\begin{eqnarray*}
			D(H_n)\ssymbol{2} =	-\frac{1}{2}L +
			\frac{4}{3(n-1)}\mathbf{w}\mathbf{w^{\prime}},
		\end{eqnarray*}
		where the rank of
		$L$ is
		$2n-3$. We also investigate the inertia of
		$D(H_n)$.
\end{abstract}


\section{Introduction}
We consider a simple connected  graph $G$ on $n$ vertices with the vertex set 
$\{v_1,v_2,\ldots,v_n\}$. In literature, several matrices have been 
associated with 
$G$ through which many structural properties of the graphs have been studied. 
A few classical graph matrices are the adjacency matrix, the Laplacian 
matrix, the distance matrix, the incidence matrix etc., 
see \cite{Bapat, Graph spectra 	book London ms}. Let us recall the 
distance matrix which is relevant to our discussion. The
 \textit{distance 
	matrix}  $D(G):=(d_{ij})$ of $G$, is an $n \times n$ symmetric matrix 
 with $d_{ij}=d(v_i,v_j)$  for all $i$ and $j$, where $d(v_i,v_j)$ denotes the
	length of a shortest path between the vertices $v_i$ and $v_j$. This matrix 
	has been widely studied in the literature and has applications in 
	chemistry, physics, computer science, 
	etc., see \cite{Bapat, Graph spectra book London ms} and the 
	references therein.  For a brief introduction, we refer to the survey 
	article
	\cite{Distance spectra survey}.

A basic problem in graph 
matrices is to give a simple formula to compute the inverses of these matrices. 
This 
problem has been extensively studied in the literature, see \cite{Balaji,Bapat, 
helm graph,Graham, Fan graph, distance well defined graph}. To state the famous
inverse formula given by Graham and 
Lov\'{a}sz 
\cite{Graham}, let us recall 
the Laplacian matrix corresponding to $G$. 
The 
\textit{Laplacian matrix} of $G$ is given by ${L}:=(l_{ij})$ where  $l_{ii}$ is
the degree of the vertex $v_i$, $l_{ij}=-1$ if
$v_i$ and $v_j$ are adjacent and zero elsewhere.  Then ${L}$ is a 
symmetric matrix of order $n$ whose row sums are zero, positive semidefinite 
and rank of ${L}$ is 
$n-1$ (see \cite{Bapat}). In \cite{Graham}, an expression for  the 
inverse of the distance matrix $D(T)$ of a tree $T$ is obtained, which is given 
by 
\begin{eqnarray}\label{Lovaz}
D(T)^{-1}=-\frac{1}{2}{L}+\frac{1}{2(n-1)}\bm{\tau\tau^{\prime}},
\end{eqnarray}
where $\mathbf{\tau}^{\prime}=\big(2-\text{deg}(v_1), 2-\text{deg}(v_2), 
\ldots, 
2-\text{deg}(v_n)\big)$ and $\text{deg}(v_i)$ is the degree of the vertex 
$v_i$. 
Motivated 
by  this,  analogous inverse formulae
for the distance 
matrices of various graphs have been established.  
For instance, the distance matrices of wheel graphs $W_n$ when $n$ is even
\cite{Balaji},  helm graphs with even number of vertices 
\cite{helm graph} and fan graphs \cite{Fan graph}. Let us mention that 
for all these matrices, the inverse formulae are obtained by replacing the 
Laplacian matrix by Laplacian-like matrix which we recall next.  A symmetric 
matrix $\bar{L}$ is said to be 
\textit{Laplacian-like} if all its 
row sums are zero \cite{distance well defined graph}. The results state that 
the inverse formulae for all the above matrices are
expressed as the 
sum of a symmetric positive semidefinite Laplacian-like matrix and a 
rank-one matrix. 
Similar inverse formulae have been determined for the distance
matrices of weighted trees, cycles, complete graphs, block graphs and bi-block 
graphs, see \cite{distance well defined graph} and the references therein for 
more graphs.

Another basic problem in this topic is to find the Moore–Penrose inverses of 
singular and 
rectangular 
graph matrices, which have been well studied in the literature (see \cite{MP 
	inverse,Balaji odd wheel graph,Bapat, sivakumar1,sivakumar2} and references 
	therein).
 Let us recall that for an  $m \times n$ real matrix 
$M$, an $n \times m$ matrix $X$ is said to be  the \textit{Moore-Penrose 
inverse} of $M$ if 
$MXM=M,~ XMX=X$,~ $(MX)^{\prime}=MX$ and $(XM)^{\prime}=XM$. It 
is known that the
Moore–Penrose inverse always exists and is unique. It is  denoted by 
$M\ssymbol{2}$. Furthermore, $M\ssymbol{2}$ coincides 
with the usual inverse $M^{-1}$ if $M$ is non-singular, see \cite{Moore inverse 
book} for more details. 
Inspired by the inverse formula result of Graham and  Lov\'{a}sz \cite{Graham}, 
the Moore-Penrose 
inverse of the distance matrix $D(W_n)$ of the wheel graph $W_n$, similar to 
(\ref{Lovaz}), was obtained for the singular case \cite{Balaji odd wheel 
graph}. More precisely, if $n$ is odd, then  the Moore-Penrose 
inverse of $D(W_n)$  is given by
\begin{equation}\label{Balaji MP formula}
D(W_n)\ssymbol{2}=
-\frac{1}{2}{{L}}+\frac{4}{n-1}\mathbf{t}\mathbf{t^{\prime}},
\end{equation}
where ${L}$ is a real symmetric positive semidefinite Laplacian-like matrix of 
order $n$  and 
$\mathbf{t} \in 
\mathbb{R}^{n}$. 

This paper focuses on the distance matrix of a helm graph $H_n$, which is 
 a generalization of star graph 
\cite{EDM}.
Several studies on helm 
	graphs have been carried out in the literature. The  
	resolving 
	domination numbers of helm graphs were analysed in \cite{domination no 
	helm}. It was proved in \cite{EDM} that the 
	distance matrix $D(H_n)$ of a 
	helm graph $H_n$ is 
	a circum 
	Euclidean distance matrix. In \cite{helm 
		graph}, it was shown that if $n$ is even then $D(H_n)$ is non-singular. 
Further, the problem of finding the inverse of $D(H_n)$ was considered and  the 
inverse formula for $D(H_n)$, similar to (\ref{Lovaz}), was provided using the 
inverse of the distance 
matrix $D(W_n)$ of the 
wheel graph $W_n$. That is,
\begin{equation}\label{Shivani formula1}
	D(H_n)^{-1}=-\frac{1}{2}\mathcal{L}+
	\frac{4}{3(n-1)}\mathbf{w}\mathbf{w}^{\prime},
\end{equation}
where $\mathcal{L}$ is a Laplacian-like  matrix and $\mathbf{w} \in 
\mathbb{R}^{2n-1}$, see Theorem $3$ in \cite{helm graph}.
Motivated by the above-mentioned results on $H_n$, our aim herein is to explore 
more results on 
$D(H_n)$, which may be helpful in studying the distance matrices of 
generalizations of 
 star graphs.

 In Section $\ref{helm graph sec 3}$, we  
show 
that $D(H_n)$ is singular if 
$n$ is odd and then derive the inertia of $D(H_n)$ by finding its rank. We 
study the Moore-Penrose inverse formula for $D(H_n)$ in Section 
$\ref{Formula for the MP inverse of helm}$. We  first establish  necessary and 
sufficient 
conditions under which  the Moore-Penrose inverses of symmetric matrices 
generally and 
$D(H_n)$ in particular, is of the form similar to (\ref{Shivani formula1}). 
Using 
these 
results, we give an alternative proof of 
the inverse formula given in (\ref{Shivani formula1}). Further, we establish an 
analogous formula for $D(H_n)\ssymbol{2}$ for the singular case. That is, if 
$n$ is odd, we 
construct a symmetric Laplacian-like matrix $L$ of order $2n-1$ and a vector 
$\mathbf{w} \in \mathbb{R}^{2n-1}$ such 
that $D(H_n)\ssymbol{2}$ is expressed as the sum of a constant multiple of $L$ 
and a rank one symmetric matrix defined by $\mathbf{w}$ (see Theorem 
\ref{Inverse Formula}). It is noteworthy to mention that our approach of 
finding 
 $D(H_n)\ssymbol{2}$ is significantly different from those of 
\cite{Balaji odd wheel graph,Balaji,helm graph,Fan graph}. Unlike \cite{helm 
graph}, the 
techniques employed to obtain a different proof for 
$D(H_n)^{-1}$ does not 
depend on $D(W_n)^{-1}$ and the proof given 
for $D(H_n)\ssymbol{2}$ formula is also  independent of $D(W_n)\ssymbol{2}$. 
 Finally, we prove that the constructed $L$ is a positive 
semidefinite matrix of rank $2n-3$ using 
the concept of simultaneous diagonalization.

\section{Preliminaries}\label{prelims}
In this section, we  fix the notations and collect some 
basic results which will be needed in this paper.  
For a matrix $M$, we denote the $i$-th row of $M$, the $j$-th column of 
$M$, the transpose of $M$, the range of $M$, the null space 
of $M$ and the rank of $M$ by $M_{i*}$, $M_{*j}$,
$M^{\prime}$, $R(M)$, $N(M)$ and $\ra(M)$ respectively. We write 
the determinant of a square matrix  $M$ as $\det(M)$ and the identity 
matrix of order $n$ as $I_n$.  The symbol 
$\diag(\mu_{1},\mu_{2},\ldots,\mu_{n})$ 
represents the $n \times n$ diagonal matrix whose $j$-th diagonal entry is 
$\mu_j$.
The 
notations $J_{n}$ and   $O_{n}$ are used to denote the matrices with 
all elements equal to $1$ and $0$ respectively. The subscripts are omitted if 
the 
{order} of the matrix is clear from the context. 
All the vectors are assumed to be column vectors, and are denoted 
by lowercase boldface letters.   We use the notations $\mathbf{e}$ 
and $\mathbf{0}$ to 
represent the vectors 
in $\mathbb{R}^{n}$ 
whose coordinates are all one and zero respectively. 
Let  $\cir(\mathbf{a^{\prime}})$ denote the 
circulant matrix of order $n$ defined by the vector 
$\mathbf{a}=(a_1, a_2, \ldots, a_{n})^{\prime} \in \mathbb{R}^{n}$ and
the notation $T_n{(2,1,1)}$ stands for the 
tridiagonal matrix of order $n$ whose diagonal entries are all $2$.
That is,
\begin{equation*}\label{C}
	\cir(\mathbf{a}^{\prime}) = 
	\begin{bmatrix}
		a_1&a_2&a_3&\cdots&a_{n-1}&a_{n}\\
		a_{n}&a_1&a_2&\cdots&a_{n-2}&a_{n-1}\\
		a_{n-1}&a_{n}&a_1&\cdots&a_{n-3}&a_{n-2}\\
		\vdots&\vdots&&\ddots&\vdots&\vdots\\
		a_3&a_4&a_5&\cdots&a_1&a_2\\
		a_2&a_3&a_4&\cdots&a_{n}&a_1
	\end{bmatrix} ~\text{and}~ T_n{(2,1,1)} =  \begin{bmatrix}
		2& 1& 0&  \cdots & 0& 0& 0\\
		1& 2& 1&  & 0& 0& 0 \\
		0& 1& 2&  & 0& 0& 0\\
		\vdots&& \ddots&\ddots&\ddots&  \\
		0& 	&&&2&1& 0\\
		0&\cdots&&	&1&2& 1\\
		0&\cdots&&	&0&1&2
	\end{bmatrix}.
\end{equation*}

Suppose that	
$A=\cir(\mathbf{x^{\prime}})$ and 
$B=\cir(\mathbf{y^{\prime}})$ where $\mathbf{x,y} \in \mathbb{R}^n$. 
Then
		\begin{equation}\label{circulant properties}
					AB= BA, \quad AB=\cir(\mathbf{x^{\prime}}B) \quad 
					\text{and}\quad 
		\cir (a \mathbf{x^{\prime}}+ b 
			\mathbf{y^{\prime}}) =aA+bB.
		\end{equation}
The above interesting 
properties of 
the circulant matrix will be used frequently in this paper.
 For more 
 results, we refer to the books \cite{Horn,Zhang}.

\section{The Distance Matrix of a Helm Graph and its Inertia}\label{helm graph 
sec 3}
We first define the  
distance matrix $D(H_n)$ of a helm graph $H_n$. Then we show that $D(H_n)$ is 
singular when $n$ is odd. Next, we derive the inertia of $D(H_n)$ after 
determining its rank.

We first recall the definition of a wheel graph. For $n\geq4$, the notation 
$C_{n-1}$ denotes the cycle of length $n-1$ and the vertices in $C_{n-1}$ are 
labelled as 
$v_1, v_2, \ldots, v_{n-1}$.
 The wheel graph $W_n$ on $n$ vertices is a graph containing the cycle 
 $C_{n-1}$ 
 and  a 
 vertex, say $v_0$, not in the 
 cycle $C_{n-1}$ which is adjacent to every vertex $v_i$ in the 
 cycle $C_{n-1}$. This paper is concerned with the helm graph which we define 
 next. The helm graph on $2n-1$ vertices, 
denoted by 
$H_n$, 
is a supergraph of $W_n$ which is obtained from $W_n$ by attaching a 
pendant vertex $u_i$ to the vertex $v_i$  lying on the outer cycle for 
all 
$i=1,2,\ldots, n-1$. The helm graph $H_7$ on $13$ vertices is given in Figure
\ref{helm graph}.
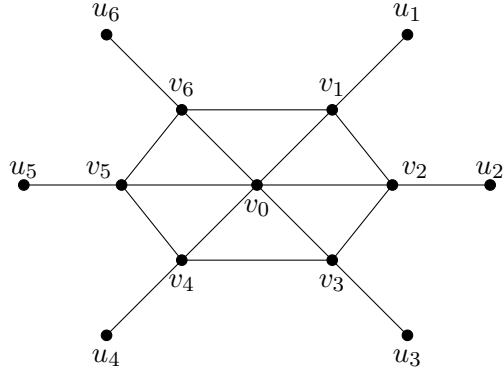
\begin{figure}[h!]
	\begin{center}
		\begin{tabular}{c}
		\begin{tikzpicture}[scale=1]
			\draw[fill=black] (0,0) circle (2pt);\node at (0,-0.3) 
			{$ 
				v_0 $};
			\draw[fill=black] (1.8,0) circle (2pt);
			\draw[fill=black] (-1.8,0) circle (2pt);
			\draw[fill=black] (3.1,0) circle (2pt);
			\draw[fill=black] (-3.1,0) circle (2pt);
			\draw[fill=black] (1,1) circle (2pt);
			\draw[fill=black] (1,-1) circle (2pt);
			\draw[fill=black] (-1,1) circle (2pt);
			\draw[fill=black] (-1,-1) circle (2pt);
			\draw[thin] 
			(0,0)--(1.8,0)--(1,1)--(-1,1)--(-1,1)--(-1.8,0)--(-1,-1)--(1,-1)--(1.8,0);
			
			\draw[thin] (-1.8,0)--(-3.1,0);
			\draw[thin] (1.8,0)--(3.1,0);
			
			\draw[fill=black] (-2,2) circle (2pt);
			\draw[fill=black] (2,2) circle (2pt);
			\draw[thin] (1,1)--(2,2);
			\draw[thin] (-1,1)--(-2,2);

			\draw[fill=black] (-2,-2) circle (2pt);
			\draw[fill=black] (2,-2) circle (2pt);
			\draw[thin] (1,-1)--(2,-2);
			\draw[thin] (-1,-1)--(-2,-2);
			
			\draw[thin] (0,0)--(-1.8,0);
			\draw[thin] (0,0)--(-1,1);
			\draw[thin] (0,0)--(-1,-1);
			\draw[thin] (0,0)--(1,1);
			\draw[thin] (0,0)--(1,-1);
			\node at (1,1.3) {$ v_1 $};
			\node at (2,2.3) {$ u_1 $};
			\node at (1,-1.3) {$ v_3 $};
			\node at (2,-2.3) {$ u_3 $};
			\node at (-1,-1.3) {$ v_4 $};
			\node at (-2,-2.3) {$ u_4 $};
			\node at (-1,1.3) {$ v_6 $};
			\node at (2.1,0.2) {$ v_2 $};
			\node at (-2.1,0.2) {$ v_5 $};
			\node at (-2,2.3) {$ u_6 $};
			\node at (3.1,0.2) {$ u_2 $};
			\node at (-3.1,0.2) {$ u_5 $};
		\end{tikzpicture}
		\end{tabular}
	\end{center}
	\caption{Helm graph $H_7$ on 13 vertices}\label{helm graph}
\end{figure}

\noindent Let $\mathbf{u}=(0,1,2,2,\ldots,2,1)^{\prime} \in \mathbb{R}^{n-1}$. 
Then the distance matrix $D(H_n)$ of the helm graph $H_n$ is given by
\begin{equation*}\label{D(H_n)}
	D(H_n)=\begin{bmatrix}
		0 & \mathbf{e}^{\prime} & 
		2\mathbf{e}^{\prime}\\[5pt]
		\mathbf{e} & \widetilde{D} & 
		\widetilde{D}+J_{n-1}\\[5pt]
		2\mathbf{e} &~ \widetilde{D}+J_{n-1} & 
		\widetilde{D}+2(J_{n-1}-I_{n-1})
	\end{bmatrix},
\end{equation*}
where $\widetilde{D}=\cir(\mathbf{u^{\prime}})$, see \cite{helm graph,EDM}.
We fix the symbol $\mathbf{s}$ to denote the vector
$(2,1,0,\ldots,0,1)^{\prime}$ in 
$\mathbb{R}^{n-1}$. By defining 
$S=\cir(\mathbf{s}^{\prime})$,
the matrix $D(H_n)$ can be re-written as
$D(H_n)=D_a+D_b$, where 
\begin{equation*}\label{D as sum of two matrices}
	D_a=\begin{bmatrix}
		0 & \mathbf{e}^{\prime} & 
		2\mathbf{e}^{\prime}\\[5pt]
		\mathbf{e} & 2J_{n-1} & 3J_{n-1}\\[5pt]
		2\mathbf{e} & 3J_{n-1} & 4J_{n-1}
	\end{bmatrix} \text{ and } 
	D_b=\begin{bmatrix}
		0 & \mathbf{0}^{\prime} & \mathbf{0}^{\prime}\\[5pt]
		\mathbf{0}  & -S & -S\\[5pt]
		\mathbf{0}  & -S & -(S+2I_{n-1}) 
		\end{bmatrix}.
	\end{equation*}
Based on this observation, it is easy to show that $D(H_n)$ is singular if $n$ 
is odd, which we prove in Theorem \ref{D is singular}.
\subsection{The Inertia of $D(H_n)$}
It has been shown that $\det(D(H_n))=3(n-1)2^{n-1}$ when $n$ is 
even (\cite{helm graph}, Theorem $2$). From numerical computations, it has 
been observed in \cite{helm graph} that $D(H_n)$ is singular if $n$  is odd. In 
the following, we give a proof of this result.
 Precisely, we show that 
$\det(H_n)=0$ if $n$ 
is odd. 
\begin{theorem}\label{D is singular}
	Let $n \geq 5$. If $n$  is an odd integer, then $D(H_n)$ is singular.
\end{theorem}
\begin{proof}
	 Define $\mathbf{v}=(1,-1,1,-1,\ldots,1,-1)^{\prime} \in 
	\mathbb{R}^{n-1}$ and $\mathbf{p}_0=(0,\mathbf{v}^{\prime}, 
	\mathbf{0}^{\prime})^{\prime} \in 
	\mathbb{R}^{2n-1}$. Then
	\begin{equation*}
		D(H_n)\mathbf{p}_0=\begin{bmatrix}
			0 & \mathbf{e}^{\prime} & 2\mathbf{e}^{\prime} \\[5pt]
			\mathbf{e} & 2J & 3J\\[5pt]
			2 \mathbf{e} & 3J & 4J
		\end{bmatrix}\begin{bmatrix}
			0 \\[5pt] \mathbf{v} \\[5pt] \mathbf{0}
		\end{bmatrix}+\begin{bmatrix}
			0 & \mathbf{0}^{\prime} & \mathbf{0}^{\prime} \\[5pt]
			\mathbf{0} & -S & -S\\[5pt]
			\mathbf{0} & -S & -(S+2I)
		\end{bmatrix}\begin{bmatrix}
		0 \\[5pt] \mathbf{v} \\[5pt] \mathbf{0}
		\end{bmatrix}
		=\begin{bmatrix}
			\mathbf{e}^{\prime}\mathbf{v} \\[5pt] 2J\mathbf{v} \\[5pt] 
			3J\mathbf{v}
		\end{bmatrix}
		+\begin{bmatrix}
			0 \\[5pt] -S\mathbf{v} \\[5pt] -S\mathbf{v}
		\end{bmatrix}.
	\end{equation*}
	Since $n-1$ is even, we have $\mathbf{e}^{\prime}\mathbf{v}=0$ which implies
	$J\mathbf{v}=\mathbf{0}$. It is easy to see that
	$S\mathbf{v}=\mathbf{0}$. Therefore,
	$D(H_n)\mathbf{p}_0=\mathbf{0}$ and hence $D(H_n)$ is singular. 
\end{proof}

Recall that the \textit{inertia} of a real symmetric 
matrix $M$ of order $n$, denoted by 
$\In(M)$, is the ordered triple $(i_+, i_{-}, i_0)$, where $i_+$, 
$i_-$  and  $i_0$ respectively denote the number of positive, 
negative, and 
zero eigenvalues of $M$ including the multiplicities. It is well known that 
$\ra(M)=i_++i_{-}$. 

It is shown in \cite{Graham and 
	Pollak} that the inertia of the distance 
matrix of any tree with $n\geq 2$ vertices is 
$(1,n-1,0)$. The inertias of the distance matrices of 
 wheel graphs and fan graphs are 
studied, see \cite{ Distance spectra survey}. The objective 
of this 
section is to find the inertia of the distance matrix of 
the helm graph.

Next, we find the rank of $D(H_n)$ which will be used to compute the 
$\In(D(H_n))$.
\begin{theorem}\label{rank of H}
Let $n\geq 4$ be an integer. Then
$\ra(D(H_n))=\begin{cases}
	2n-1 & \text{if $n$ is even},\\
	2n-2 & \text{if $n$ is odd}.
\end{cases}$
\end{theorem}
\begin{proof}
If $n$ is even, then $\det(D(H_n))=3(n-1)2^{n-1}\neq 0$ by Theorem 2 in 
\cite{helm graph}. Hence $\ra(D(H_n))=2n-1$. Assume that $n$ is odd. We claim 
that 
$\ra(D(H_n))=2n-2$. To prove the claim,  it is enough to show 
$N(D(H_n))=\s\{\mathbf{z_0}\}$ for 
some non-zero $\mathbf{z_0} \in \mathbb{R}^{2n-1}$. Let 
$\mathbf{z}=(\gamma,\mathbf{p}^{\prime},\mathbf{q}^{\prime})^{\prime} \in 
\mathbb{R}^{2n-1}$ where 
$\mathbf{p},\mathbf{q} \in \mathbb{R}^{n-1}$ and $\gamma\in \mathbb{R}$. 
Suppose 
$D(H_n)\mathbf{z}=\mathbf{0}$. Then we have the following system of equations:
\begin{align}
	\mathbf{e^{\prime}}\mathbf{p}+2\mathbf{e^{\prime}}\mathbf{q}&=0.
	\label{Dz first}\\
	\gamma\mathbf{e}+2J\mathbf{p}-S\mathbf{p}+3J\mathbf{q}-S\mathbf{q}&=\mathbf{0}.
	 \label{Dz second}\\
	2\gamma\mathbf{e}+3J\mathbf{p}-S\mathbf{p}+4J\mathbf{q}-S\mathbf{q}&=2\mathbf{q}.
	 \label{Dz third}
\end{align}
Subtracting (\ref{Dz second}) from (\ref{Dz third}), we obtain 
$\gamma\mathbf{e}+J\mathbf{p}+J\mathbf{q}=2\mathbf{q}$ which can be written as 
\begin{equation}\label{Dz four}
(\gamma+\mathbf{e}^{\prime}\mathbf{p}+\mathbf{e}^{\prime}\mathbf{q})\mathbf{e}=
2\mathbf{q}.
\end{equation}
From (\ref{Dz first}) and (\ref{Dz four}), we get 
$(\gamma-\mathbf{e}^{\prime}\mathbf{q})\mathbf{e}=2\mathbf{q}$. This implies 
$(\gamma-\mathbf{e}^{\prime}\mathbf{q})\mathbf{e}^{\prime}\mathbf{e}=
2\mathbf{e}^{\prime}\mathbf{q}$ and hence 
\begin{equation}\label{Dz five}
\gamma=\frac{n+1}{n-1}\mathbf{e}^{\prime}\mathbf{q}.
\end{equation}
Now we obtain  another expression for $\gamma$ by premultiplying (\ref{Dz 
second}) by $\mathbf{e}^{\prime}$. Using the fact 
$\mathbf{e}^{\prime}S=4\mathbf{e}^{\prime}$, we get 
$\gamma(n-1)+2(n-1)\mathbf{e}^{\prime}\mathbf{p}-4\mathbf{e}^{\prime}\mathbf{p}
+3(n-1)\mathbf{e}^{\prime}\mathbf{q}-4\mathbf{e}^{\prime}\mathbf{q}=0$. By 
(\ref{Dz first}), the above equation reduces to
\begin{equation}\label{Dz six}
	\gamma=\frac{n-5}{n-1}\mathbf{e}^{\prime}\mathbf{q}.
\end{equation}
Comparing (\ref{Dz five}) and (\ref{Dz six}) gives 
$\mathbf{e}^{\prime}\mathbf{q}=0$ 
which implies $\gamma=0$. It
follows 
from 
(\ref{Dz first}) that  $\mathbf{e}^{\prime}\mathbf{p}=0$  and hence by (\ref{Dz 
four}), $\mathbf{q}=\mathbf{0}$. Thus, 
$\mathbf{z}=(0,\mathbf{p}^{\prime},\mathbf{0}^{\prime})^{\prime}$ and the 
system $D(H_n)\mathbf{z}=\mathbf{0}$ reduces to $S\mathbf{p}=\mathbf{0}$. 
Note 
that $S$ can be written as 
\begin{equation*}\label{S}
S=\begin{bmatrix}
		2 & \mathbf{e_1}^{\prime}+ \mathbf{e_{n-2}}^{\prime}\\[5pt]
		\mathbf{e_1}+\mathbf{e_{n-2}}  & T_{n-2}(2,1,1)
	\end{bmatrix}.
\end{equation*}
By Theorem $5.5$ in \cite{Zhang}, $\det(T_{n-2}(2,1,1))=n-1$ which gives 
$\ra(S) \geq n-2$. Note that the order of $S$ is $n-1$ and 
$S\mathbf{v}=\mathbf{0}$ where 
$\mathbf{v}=(1,-1,1,-1,\ldots,1,-1)^{\prime} \in \mathbb{R}^{n-1}$. Hence 
$\ra(S)= n-2$ 
and 
$N(S)=\s\{\mathbf{v}\}$. This implies $N(D(H_n))=\s\{\mathbf{z_0}\}$ where 
$\mathbf{z_0}=(0,\mathbf{v}^{\prime},\mathbf{0}^{\prime})^{\prime}$. Hence the 
proof.
\end{proof}

\begin{theorem}\label{inertia}
	Let $n\geq 4$. The inertia of $D(H_n)$ is 
	$	\In D(H_n)=
	\begin{cases} (1,2n-2,0) & \text{ if $n$ is even}, \\
		(1,2n-3,1) & \text{ if $n$ is odd}.
	\end{cases}$
\end{theorem}
\begin{proof}
	It is proved that  $D(H_n)$ is an Euclidean distance matrix 
	(\cite{EDM},{Theorem 
		14}). 
	Therefore, $D(H_n)$ has exactly one positive eigenvalue ({see 
		Theorem 2 in \cite{EDM}}), that is $n_+(D(H_n))=1$ for all $n\geq 4$. 
		By 	Theorem 
		\ref{rank of H}
		and using the fact that 
	$\ra(D(H_n))=	n_+ (D(H_n))+	n_- (D(H_n))$, we obtain the  desired 
	result.
\end{proof}
\section{Formulae for the inverse and the Moore-Penrose inverse of $D(H_n)$}
\label{Formula for the MP inverse of helm}
It is shown in \cite{helm 	graph} that if $n\geq 4$ is an even 
integer, then $D(H_n)$ is non-singular and a formula for 
$D(H_n)^{-1}$ has been obtained using the inverse of the distance matrix of the 
wheel graph $W_n$. Also, the formula is expressed as a constant 
multiple of a symmetric Laplacian-like matrix $\mathcal{L}$  plus a symmetric 
rank-one 
matrix defined by a vector 
$\mathbf{w}$. To be more precise,
\begin{equation}\label{Shivani formula}
	D(H_n)^{-1}=-\frac{1}{2}\mathcal{L}+
	\frac{4}{3(n-1)}\mathbf{w}\mathbf{w}^{\prime},
\end{equation}
where $\mathcal{L}$ is a symmetric positive semidefinite matrix of 
rank $2n-2$ 
and $\mathbf{w}=\frac{1}{4}(5-n,-\mathbf{e^{\prime}},
2\mathbf{e^{\prime}})^{\prime} \in \mathbb{R}^{2n-1}$ (\cite{helm 
graph}, Theorem $3$). In Section $\ref{short proof}$, we give  another proof of 
this 
result. We 
point out that our proof does not depend on $D(W_n)^{-1}$. 

Motivated by the inverse formula given in (\ref{Shivani formula}), we study an 
analogous formula for the Moore-Penrose inverse of $D(H_n)$ when $D(H_n)$ is
singular, see Section \ref{Moore-Penrose inverse}.
If $n$ is odd, we proved that $D(H_n)$ is singular (Theorem \ref{D is 
singular}). In this case, we establish a formula for $D(H_n)\ssymbol{2}$ by 
finding an equivalent formulations for the general matrix case and $D(H_n)$ 
which will be discussed in Section $\ref{charc}$.
\subsection{Characterizations  for the Moore-Penrose inverse of a general 
symmetric matrix and $D(H_n)$}\label{charc}
In this subsection, we derive a necessary and sufficient condition for a 
symmetric matrix $D$ to have its Moore-Penrose inverse of the form 
(\ref{Shivani formula}). The conditions are given in terms of system of linear 
equations and matrix equations, where the precise statement is given below. 
\begin{theorem}\label{eqiv formulation}
	Let $D$ be a symmetric matrix of order $n$ with $\mathbf{e} \in R(D)$  and 
	$\mathbf{w} \in \mathbb{R}^n$ such that 
	$\mathbf{e^{\prime}} \mathbf{w}=1$. Suppose that $L$ is a symmetric 
	Laplacian-like matrix  and $\alpha$ is a non-zero real number.
	Then $D^{\ssymbol{2}} =-\frac{1}{2} 
	L+\alpha \mathbf{w} \mathbf{w^{\prime}}$ if and only 
	if 
	$D\mathbf{w} =\frac{1}{\alpha} \mathbf{e}$ and $LD+2I=2\mathbf{w} 
	\mathbf{e^{\prime}}+\widetilde{V}$ for some symmetric matrix
	$\widetilde{V}$ satisfying 
	$D\widetilde{V}=O$ and $\widetilde{V}(-\frac{1}{2} L+\alpha \mathbf{w} 
	\mathbf{w^{\prime}}) =O$. 
\end{theorem}
\begin{proof}
	Assume that $D^{\ssymbol{2}} =-\frac{1}{2} L+\alpha \mathbf{w} 
	\mathbf{w^{\prime}}$ with $L\mathbf{e}=\mathbf{0}$.
	Premultiplying $D^{\ssymbol{2}}$ by $D$, we get $DD^{\ssymbol{2}}= 
	-\frac{1}{2} DL+\alpha D\mathbf{w} \mathbf{w^{\prime}}$. Since  $\mathbf{e} 
	\in R(D)$, we have $DD^{\ssymbol{2}} \mathbf{e}=\mathbf{e}$ (see 
	\cite{Moore inverse book}).
	Therefore, $ DD^{\ssymbol{2}} \mathbf{e}=-\frac{1}{2} DL\mathbf{e}+\alpha D 
	\mathbf{w} \mathbf{w^{\prime}} \mathbf{e} = \alpha D\mathbf{w}$ which gives 
	$ D\mathbf{w} =\frac{1}{\alpha} \mathbf{e}$.
	Now postmultiplying $D^{\ssymbol{2}}$ by $D$, we obtain $D^{\ssymbol{2}} 
	D=-\frac{1}{2} LD+\alpha \mathbf{w} \mathbf{w^{\prime}} D=-\frac{1}{2} 
	LD+\mathbf{w}\mathbf{e^{\prime}}$. This implies 
	$LD+2I=2\mathbf{w}\mathbf{e^{\prime}}+2(I-D^{\ssymbol{2}} D) = 
	2\mathbf{w}\mathbf{e^{\prime}}+\widetilde{V}$, where 
	$\widetilde{V}:=2(I-D^{\ssymbol{2}} D)$.
	Note that $\widetilde{V}$ is a symmetric matrix as $D^{\ssymbol{2}} D$ is 
	symmetric.
	Clearly $D\widetilde{V}=O$. We claim that $\widetilde{V} (-\frac{1}{2} 
	L+\alpha 
	\mathbf{w} \mathbf{w^{\prime}})=O$. We have 
	$\widetilde{V}D\ssymbol{2}=(2I-2D\ssymbol{2}D)D\ssymbol{2}=2D\ssymbol{2}-
	2D\ssymbol{2}DD\ssymbol{2}=O$.
 By the assumption on $D\ssymbol{2}$, the claim follows.
	Conversely assume that 
	\begin{equation}\label{identity}
		D\mathbf{w} 
	=\frac{1}{\alpha} \mathbf{e} ~\text{and}~ LD+2I=2\mathbf{w} 
	\mathbf{e^{\prime}}+\widetilde{V} ,
	\end{equation}
 where $L$ and $ \widetilde{V}$ are symmetric 
	matrices such that $L\mathbf{e}=\mathbf{0}$, 
	$D\widetilde{V}=O$ and $\widetilde{V} (-\frac{1}{2} L+\alpha \mathbf{w} 
	\mathbf{w^{\prime}}) =O$.
	Let $X=-\frac{1}{2} L+\alpha \mathbf{w} \mathbf{w^{\prime}}$. We claim that 
	$X$ is the Moore-Penrose inverse of $D$. First we show that $XD$ is 
	symmetric. We have
	$XD =-\frac{1}{2} LD+\alpha \mathbf{w} \mathbf{w^{\prime}}D$. Using 
(\ref{identity}), we can write
	$XD=I-\frac{1}{2}\widetilde{V}$. Therefore, $XD$ is symmetric. Also $DX$ is 
	a 
	symmetric matrix which follows from the fact that $X$, $D$ and $XD$ are 
	symmetric. Since $D\widetilde{V}=O$, it is easy to see that 
	$DXD=D(I-\frac{1}{2}\widetilde{V})=D$. The 
	assumption $\widetilde{V} X=O$ gives $XDX=X$. This completes the proof.	
\end{proof}
The next theorem gives the uniqueness of the Laplacian-like matrix $L$ and 
the vector $\mathbf{w}$, which satisfy the identities in (\ref{identity}).
\begin{theorem}\label{uniqueness of L}
	Let $D$ be a symmetric matrix of order $n$. If there exist a non-zero real 
	number $\alpha$ and $\mathbf{w} \in\mathbb{R}^n$ with 
	$\mathbf{e^{\prime}}\mathbf{w}=1$ such that $D^{\ssymbol{2}} =-\frac{1}{2} 
	L+\alpha \mathbf{w} \mathbf{w^{\prime}}$ for some Laplacian-like matrix 
	$L$, then the scalar $\alpha$, the vector $\mathbf{w}$ and the matrix $L$ 
	are unique.
\end{theorem}
\begin{proof}
	Suppose $D^{\ssymbol{2}} =-\frac{1}{2} L_0+\beta \mathbf{z} 
	\mathbf{z^{\prime}}$ for some non-zero $\beta \in \mathbb{R}$, 
	$\mathbf{z}\in \mathbb{R}^n$ with $\mathbf{e^{\prime}}\mathbf{z}=1$ and 
	$L_0\mathbf{e}=\mathbf{0}$.
	Then $\mathbf{e^{\prime}} D^{\ssymbol{2}} \mathbf{e} 
	=\mathbf{e^{\prime}}(\alpha \mathbf{w} \mathbf{w^{\prime}}) \mathbf{e} 
	=\mathbf{e^{\prime}}(\beta \mathbf{z} \mathbf{z^{\prime}}) \mathbf{e}$. 
	This implies $\alpha=\beta$ and hence $ D^{\ssymbol{2}} =-\frac{1}{2} 
	L+\alpha \mathbf{w} \mathbf{w^{\prime}} =-\frac{1}{2} L_0+\alpha \mathbf{z} 
	\mathbf{z^{\prime}}$. By finding $D^{\ssymbol{2}} \mathbf{e}$, we get $ 
	\mathbf{w}=\mathbf{z}$. Hence $L=L_0$ and hence the 
	proof.
\end{proof}
In the following remark, we mention certain facts about the system 
$D(H_n)\mathbf{w}=\frac{1}{\alpha} \mathbf{e}$  where $\alpha$ is a non-zero 
real number.
\begin{remark}\label{w from jacklick}
	It was proved in (\cite{EDM}, Theorem 15) that the system 	
	$D(H_n)\mathbf{w}=\frac{3(n-1)}{4} \mathbf{e}$ has a solution 
	$\mathbf{w_0}=\frac{1}{4}(5-n,-\mathbf{e^{\prime}}, 
2\mathbf{e^{\prime}})^{\prime} \in \mathbb{R}^{2n-1}$. Note that   
$\mathbf{e^{\prime}}\mathbf{w_0}=1$ and $\mathbf{e} \in R(D(H_n))$. In fact, it 
can be shown that if the system $D(H_n)\mathbf{w}=\frac{1}{\alpha} \mathbf{e}$ 
has a solution 	$\mathbf{w} \in \mathbb{R}^{2n-1}$ with 
$\mathbf{e^{\prime}}\mathbf{w}=1$ then 	$\alpha =\frac{4}{3(n-1)}$. We omit the 
proof as it is similar to the case of $D(H_n)\mathbf{z}=\mathbf{0}$, 
(see Theorem \ref{rank of H}). Moreover, if $D(H_n)$ is singular then the 
solution set of $D(H_n)\mathbf{w}=\frac{3(n-1)}{4} \mathbf{e}$ is 
$\mathbf{w_0}+N(D(H_n))=\left\{\frac{1}{4}(5-n,\beta\mathbf{v}^{\prime}
-\mathbf{e^{\prime}}, 2\mathbf{e^{\prime}})^{\prime}\in 	\mathbb{R}^{2n-1}: 
\beta \in \mathbb{R}\right\}$ where 
$\mathbf{v}=(1,-1,1,-1,\ldots,1,-1)^{\prime}\in \mathbb{R}^{n-1}$ (see Theorem 
\ref{rank of H}). Of course, the solution is unique if $D(H_n)$ is non-singular.
\end{remark}
The next result gives a  necessary 
condition on $\widetilde{V}$  which will be useful in obtaining the  
identities
 (\ref{identity}) in the case of $D(H_n)$. 
\begin{lemma}\label{Necessary condition on V}
		Let 	$S=\cir(\mathbf{s}^{\prime})$ where
	$\mathbf{s}=(2,1,0,\ldots,0,1)^{\prime} \in \mathbb{R}^{n-1}$ and let $D$ 
	be the distance matrix of $H_n$. If there exists a symmetric matrix 
	$\widetilde{V}$ such that $D\widetilde{V}=O$, then 
	$\widetilde{V}=\left[\begin{smallmatrix} 0 && \mathbf{0^{\prime}} && 
	\mathbf{0^{\prime}}\\
		\mathbf{0} && E && O\\
		\mathbf{0} && O && O
	\end{smallmatrix}\right]$ for some symmetric matrix $E$ of order $n-1$ 
	with $E	\mathbf{e}=\mathbf{0}$ and $SE=O$.
\end{lemma}
\begin{proof}
	 Suppose
		 $ \widetilde{V}=\left[\begin{smallmatrix} \gamma && 
		 \mathbf{p^{\prime}} && 
		\mathbf{q^{\prime}}\\
		\mathbf{p} && E && F\\
		\mathbf{q} && F^{\prime} && G
	\end{smallmatrix}\right]$ where $\gamma \in \mathbb{R}$,  $	\mathbf{p},	
	\mathbf{q} \in 
\mathbb{R}^{n-1}$ and $E,F,G$ are  matrices of 
	order $n-1$. Assume that $D\widetilde{V}=O$. Then 
	$D\left[\begin{smallmatrix} 
		\gamma \\
		\mathbf{p}\\
		\mathbf{q} 
	\end{smallmatrix}\right] = \mathbf{0}$ gives that
 $\gamma=0$ and $	\mathbf{q}=\mathbf{0}$ (see the proof of Theorem \ref{rank 
 of H}). Similarly, we conclude that the first and the last $(n-1)$ 
 co-ordinates 
 of each column of $\widetilde{V}$ are zero. That is, $\mathbf{p}=\mathbf{0}$, 
 $F^{\prime}=O$ and 
 $G=O$. Now $D\widetilde{V}=O$ reduces to $\mathbf{e}^{\prime}E=O$ and 
 $(2J-S)E=O$. Thus, $E	\mathbf{e}=\mathbf{0}$ and $SE=O$. 
\end{proof}

We now derive a necessary and sufficient condition on the Laplacian-like matrix 
$L$ such that 
the  identities in (\ref{identity}) are satisfied. Recall
 that for a real  matrix $M$, 
$R(M^{\prime})=R(M\ssymbol{2})$ (see \cite{Moore inverse book}), which will be 
used 
in the 
next proof.

\begin{theorem}\label{necessary and sufficient conditions on L}
	Suppose that 	$S=\cir(\mathbf{s^{\prime}})$ where 
	$\mathbf{s}=(2,1,0,\ldots,0,1)^{\prime}\in \mathbb{R}^{n-1}$. Let 
  $D$ be the distance matrix of $H_n$. Then
	$D^{\ssymbol{2}} =-\frac{1}{2} L+\alpha \mathbf{w} \mathbf{w^{\prime}}$ for 
	some Laplacian-like matrix $L$, non-zero $\alpha \in \mathbb{R}$ and 
	$\mathbf{w}\in 
	\mathbb{R}^{2n-1}$ with $\mathbf{e^{\prime}}\mathbf{w}=1$ if and only if 
	$L=\left[\begin{smallmatrix} \frac{n-1}{2} & 
	-\frac{1}{2}\mathbf{e^{\prime}} & 
	\mathbf{0^{\prime}}\\[3pt]
		-\frac{1}{2}\mathbf{e} & A & B\\[3pt]
		\mathbf{0} & B^{\prime} & I_{n-1}\\
	\end{smallmatrix}\right]$ where $A$ and $B$ are symmetric matrices of order 
	$n-1$
	satisfying the following conditions:
\begin{multicols}{3}
	\begin{itemize}
		\item[(i)] $A\mathbf{e}=\frac{3}{2} \mathbf{e}$
		\item[(ii)] $B\mathbf{e}= \mathbf{-e}$
		\item[(iii)] $BS=-S$
		\item[(iv)] $(B+I)A=O$
		\item[(v)] $(B+I)B=O$
		\item[(vi)] $(A+B)S+2B=O$.
	\end{itemize}
\end{multicols} 
\end{theorem}
\begin{proof}
	Assume that
	$	D^{\ssymbol{2}} =-\frac{1}{2} L+\alpha \mathbf{w} \mathbf{w^{\prime}}  
		\text{ with } L\mathbf{e}=\mathbf{0} \text{ and } 
		\mathbf{e^{\prime}}\mathbf{w}=1$. 
Note that $D\ssymbol{2}$ is symmetric as $D$ is symmetric \cite{Moore inverse 
book}. This implies $L$ is 
symmetric. We first claim that 	
$\mathbf{w}=\frac{1}{4}(5-n,-\mathbf{e^{\prime}},
2\mathbf{e^{\prime}})^{\prime}$.
	By Theorem \ref{eqiv formulation} and Remark \ref{w from jacklick}, we have
	$D\mathbf{w}=\frac{1}{\alpha}\mathbf{e}$ where $\alpha =\frac{4}{3(n-1)}$. 
	If $n$ is even then $D$ is non-singular and $\mathbf{w}$ is the unique 
	solution of $D\mathbf{w}=\frac{1}{\alpha}\mathbf{e}$. Now assume that $n$ 
	is odd. Then from Remark \ref{w from jacklick}, 
	$\mathbf{w}=\frac{1}{4}(5-n,\beta\mathbf{v^{\prime}}-\mathbf{e^{\prime}},
	2\mathbf{e^{\prime}})^{\prime}$ for some $\beta\in \mathbb{R}$ and 
	$\mathbf{v}=(1,-1,1,-1,\ldots,1,-1)^{\prime}\in \mathbb{R}^{n-1}$. Note 
	that 	$	D^{\ssymbol{2}}\mathbf{e} =-\frac{1}{2} L\mathbf{e}+\alpha 
	\mathbf{w} \mathbf{w^{\prime}}\mathbf{e}=\alpha\mathbf{w}$. Thus  
	$\mathbf{w} \in R(D\ssymbol{2})=R(D)$ because $D$ is symmetric.
Since 
	$(0,\mathbf{v^{\prime}},\mathbf{0^{\prime}})^{\prime}\in 
	N(D)$, we get 
	$(0,\mathbf{v^{\prime}},\mathbf{0^{\prime}})\mathbf{w}=0$. That 
	is, $\frac{1}{4}\mathbf{v^{\prime}}(\beta\mathbf{v}-\mathbf{e})
	=\frac{1}{4}[\beta(n-1)-\mathbf{v^{\prime}}\mathbf{e}]=
	\frac{1}{4}\beta(n-1)=0$. Therefore, $\beta=0$ and hence  
	$\mathbf{w}=\frac{1}{4}(5-n,-\mathbf{e^{\prime}},
	2\mathbf{e^{\prime}})^{\prime}$.
	Suppose $L=\left[\begin{smallmatrix} \gamma & \mathbf{p^{\prime}} & 
	\mathbf{q^{\prime}}\\
		\mathbf{p} & A & B\\
		\mathbf{q} & B^{\prime} & C
	\end{smallmatrix}\right]$ where $\gamma \in \mathbb{R}$,  $	\mathbf{p},	
\mathbf{q} \in 
\mathbb{R}^{n-1}$ and $A$ and $B$ are  matrices of 
order $n-1$.
		Then the condition $L\mathbf{e}=\mathbf{0}$ gives the following three 
		equations:
	\begin{align}
		\gamma+\mathbf{p^{\prime}}\mathbf{e}+\mathbf{q^{\prime}}\mathbf{e}&=0 .
		\label{Le first}\\
		\mathbf{p}+A\mathbf{e}+B\mathbf{e}&=\mathbf{0}. \label{Le second}\\
		\mathbf{q}+B^{\prime}\mathbf{e}+C\mathbf{e}&=\mathbf{0}. \label{Le 
		third}
	\end{align}
By Theorem \ref{eqiv formulation}, 
	$LD=2\mathbf{w}
	\mathbf{e^{\prime}}-2I+\widetilde{V}$ with $D\widetilde{V}=O$ and 
	$\widetilde{V}(-\frac{1}{2} L+\alpha \mathbf{w}
	\mathbf{w^{\prime}}) =O$ where $\widetilde{V}$ is a symmetric matrix.
	Then by Lemma \ref{Necessary condition on V}, there exists an $(n-1)\times 
	(n-1)$ 
	symmetric matrix $X$ such that
	\begin{equation}\label{visualize of V}
		\widetilde{V}=\begin{bmatrix} 0 & \mathbf{0^{\prime}} & 
		\mathbf{0^{\prime}}\\
			\mathbf{0} & X & O\\
			\mathbf{0} & O & O
		\end{bmatrix} \text{  with } X 
		\mathbf{e}=\mathbf{0}.
	\end{equation}
Note that
	\begin{equation}\label{eqnH7}
		2\mathbf{w} \mathbf{e^{\prime}}-2I+\widetilde{V}=\begin{bmatrix}
			\frac{1-n}{2} & \frac{5-n}{2}\mathbf{e^{\prime}} &
			\frac{5-n}{2}\mathbf{e^{\prime}}\\[10pt]
			-\frac{1}{2}\mathbf{e}  & -\frac{1}{2}J-2I+X & -\frac{1}{2}J\\[10pt]
			\mathbf{e}  & J & J-2I
		\end{bmatrix}=LD.
	\end{equation}
Suppose
\begin{equation}\label{expression of LD}
	LD=Y=
	\begin{bmatrix}
		Y_{11}& Y_{12}& Y_{13}\\[1pt]
		Y_{21}& Y_{22}& Y_{23}\\[1pt]
		Y_{31}& Y_{32}& Y_{33}
	\end{bmatrix}
\end{equation}
where $Y$ is partitioned similarly to $L$.
	Equating the $(1,2)^{\text{th}}$ block of (\ref{eqnH7})
	and (\ref{expression of LD}) ,
	\begin{align}
	Y_{12}=	
	\gamma\mathbf{e^{\prime}}+2\mathbf{p^{\prime}}J-\mathbf{p^{\prime}}S+3\mathbf{q^{\prime}}J-\mathbf{q^{\prime}}S
		 &=\frac{5-n}{2}\mathbf{e^{\prime}}. \label{eqnH8}
	\end{align}
Similarly $(1,3)^{\text{rd}}$ block gives
	\begin{align}
Y_{13}=		
2\gamma\mathbf{e^{\prime}}+3\mathbf{p^{\prime}}J-\mathbf{p^{\prime}}S+4\mathbf{q^{\prime}}J-\mathbf{q^{\prime}}S-2\mathbf{q^{\prime}}
	&=\frac{5-n}{2}\mathbf{e^{\prime}}. \label{eqnH9}
\end{align}
	Subtracting (\ref{eqnH8}) from (\ref{eqnH9}) gives
	$\gamma\mathbf{e^{\prime}}+\mathbf{p^{\prime}}J
	+\mathbf{q^{\prime}}J=2\mathbf{q^{\prime}}$.
	 This implies that
	$2q_i=\gamma+\mathbf{p^{\prime}}\mathbf{e}+\mathbf{q^{\prime}}\mathbf{e}$ 
	for all $i=1,2,\ldots,n-1$ where 
	$\mathbf{q}=(q_1,q_2,\ldots,q_{n-1})^{\prime} \in 
	\mathbb{R}^{n-1}$. 
	Using (\ref{Le first}), we get $q_i=0$, for all $i$ 
	and hence $\mathbf{q}=\mathbf{0}$. We have 
	$C\mathbf{e}=-B^{\prime}\mathbf{e}$ from 
	(\ref{Le third}).
	Equating $(2,1)^{\text{th}}$ and $(3,1)^{\text{th}}$ blocks of (\ref{eqnH7})
and (\ref{expression of LD}), we get
	\begin{equation}\label{eqnH10}
	Y_{21}=			A\mathbf{e}+2B\mathbf{e} =-\frac{1}{2}\mathbf{e}  ~~~ 
		\text{and}~~~
	Y_{31}=			B^{\prime}\mathbf{e}+2C\mathbf{e} =\mathbf{e}. 
	\end{equation}
	Using $C\mathbf{e}=-B^{\prime}\mathbf{e}$ in (\ref{eqnH10}) gives 
	$C\mathbf{e}=\mathbf{e}$ and hence $B^{\prime}\mathbf{e}=-\mathbf{e}$. 
Now we show that $B$ is symmetric. Substituting the first 
	equation of (\ref{eqnH10}) in 
	(\ref{Le second}), we get $\mathbf{p}-B\mathbf{e}=\frac{1}{2}\mathbf{e}$. 
	Postmultiplying this equation by $\mathbf{e^{\prime}}$ gives
	$\mathbf{p}\mathbf{e^{\prime}}=BJ+\frac{1}{2}J$. Also, from the first 
	equation of (\ref{eqnH10}), we have $AJ=-2BJ-\frac{1}{2}J$. Note that  
	$Y_{22}=	
	\mathbf{p}\mathbf{e^{\prime}}+A(2J-S)+B(3J-S)$ and  $Y_{23}=	
	2\mathbf{p}\mathbf{e^{\prime}}+A(3J-S)+B(4J-S-2I)$.
	Comparing these with the corresponding 
	blocks of (\ref{eqnH7}) and using $\mathbf{p}\mathbf{e^{\prime}}$ and $AJ$, 
	we get $AS+BS=2I-X$ and $AS+BS+2B=O$.
	 This 
	implies $(vi)$ is proved and $2B=X-2I$. Hence $B$ is symmetric as $X$ is 
	symmetric.
	Therefore, 
	$A\mathbf{e}=-2B^{\prime}\mathbf{e} -\frac{1}{2}\mathbf{e}= 
	\frac{3}{2}\mathbf{e}$ 
	and $\mathbf{p}=\frac{-1}{2}\mathbf{e}$ 
	follows from (\ref{eqnH10}) and (\ref{Le second}) respectively. Also by 
	(\ref{Le first}), $\gamma =-\mathbf{p^{\prime}}\mathbf{e}=\frac{n-1}{2}$.
		Substituting $BJ=-J$, $CJ=J$ and $\mathbf{q}=\mathbf{0}$ in 
	(\ref{expression of LD}), and then equating $(3,2)^{\text{th}}$ and 
	$(3,3)^{\text{th}}$ blocks of (\ref{expression of LD}) and (\ref{eqnH7}), 
	we 
	have $BS+CS=O$ and $BS+CS+2C=2I$ which implies $C=I$ and $BS=-S$. 	This 
	proves $(iii)$. To complete the only  if part of the proof, it is 
	remaining  to show 
	$(B+I)A=O$ and 
	$(B+I)B=O$.
	It is easy to verify that $\widetilde{V}\mathbf{w}=\mathbf{0}$ where 
	$\widetilde{V}$ is given in (\ref{visualize of V}) with $X=2(B+I)$. Then 
	the conditions $(B+I)A=O$ and $(B+I)B=O$ are easily derived from 
	$\widetilde{V}(-\frac{1}{2} L+\alpha \mathbf{w}
	\mathbf{w^{\prime}}) =O$.
	
	Conversely, assume that $L=\left[\begin{smallmatrix}
		\frac{n-1}{2} & \frac{-1}{2}\mathbf{e^{\prime}} &
		\mathbf{0^{\prime}}\\
		\frac{-1}{2}\mathbf{e}  & A & B\\
		\mathbf{0}  & B^{\prime} & I
	\end{smallmatrix}\right]$ where the symmetric matrices $A$ and $B$ satisfy 
	the 
	conditions 
	$(i)-(vi)$. Using the conditions $(i)$ and $(ii)$, it is easy to see that 
	$L\mathbf{e}=\mathbf{0}$. Note that 
	$D\mathbf{w}=\frac{1}{\alpha}\mathbf{e}$ where
	$\mathbf{w}=\frac{1}{4}(5-n,-\mathbf{e^{\prime}},
	2\mathbf{e^{\prime}})^{\prime}$ and $\alpha =\frac{4}{3(n-1)}$. We claim 
	that $D^{\ssymbol{2}} =-\frac{1}{2}
	L+\alpha \mathbf{w} \mathbf{w^{\prime}}$. By  
	Theorem \ref{eqiv formulation}, it is enough to find a symmetric 
	matrix $\widetilde{V}$ 
	of order $2n-1$ satisfying 
	$LD+2I=2\mathbf{w} \mathbf{e^{\prime}}+\widetilde{V}$ with 
	$D\widetilde{V}=O$ and $\widetilde{V}(-\frac{1}{2} L+\alpha \mathbf{w}
	\mathbf{w^{\prime}}) =O$. 
 Choose $ 
	\widetilde{V}=\left[\begin{smallmatrix} 0 & \mathbf{0^{\prime}} & 
	\mathbf{0^{\prime}}\\
		\mathbf{0} & 2(B+I) & O\\
		\mathbf{0} & O & O
	\end{smallmatrix}\right]$. Then it is easy to see that $LD+2I=2\mathbf{w} 
	\mathbf{e^{\prime}}+\widetilde{V}$. From the assumptions $(i) $ to $(v)$, 
	it is clear that $D\widetilde{V}=O$, $\widetilde{V}L=O$  
	and $\widetilde{V}\mathbf{w}=\mathbf{0}$. Hence $\widetilde{V}(-\frac{1}{2} 
	L+\alpha 
	\mathbf{w}
	\mathbf{w^{\prime}}) =O $. This completes the proof.
	
\end{proof}

\subsection{A short proof of an inverse formula for $D(H_n)$}\label{short 
proof}
If $n$ is even, an inverse formula for $D(H_n)$ is given as the sum of 
symmetric Laplacian-like matrix and a rank one matrix (\cite{helm graph}, 
Theorem $3$). While deriving this, the inverse of the 
distance matrix $D(W_n)$ of a wheel graph $W_n$ is used. In 
this section, we offer a different 
proof of this result without employing any results pertaining to $D(W_n)^{-1}$.

We use circulant matrices of 
order $n-1$ while introducing the Laplacian-like matrices. The vectors 
defining 
the circulant matrices follow certain type of 
symmetry in the last $n-2$ coordinates, which we recall below.

\begin{definition}[\hspace{1sp}\cite{Balaji odd wheel graph,Balaji}]
Let $n \geq 4$.	A vector $\mathbf{z}=(z_1,z_2, \ldots, z_{n-1})^{\prime} \in 
\mathbb{R}^{n-1}$ is 
	said to follow 
	symmetry 
	in its 
	last $n-2$ coordinates if $z_i=z_{n+1-i} \text{ for 
		all } i=2,3,\ldots,n-1$. 
	\end{definition}
Let $\Delta$ be the collection of all vectors $\mathbf{z}\in \mathbb{R}^{n-1}$ 
that follow
symmetry in its 
last $n-2$ coordinates. That is,
$\Delta:=\{\mathbf{z}=(z_1,z_2,\ldots,z_{n-1})^{\prime}:z_i=z_{n+1-i}
\text{ 	for all } i=2,3,\ldots,n-1\}.$
\begin{remark}\label{symmetry of circulant matrices} 
	The above-defined $\Delta$ is a subspace of $\mathbb{R}^{n-1}$.
	It is observed in \cite{EDM} (see Theorem 10) that if $\mathbf{c}\in 
	\Delta$ 
	and $C=\cir(\mathbf{c}^{\prime})$, 
	then $C$ is a symmetric matrix.
\end{remark}
The following lemma is useful in computing the vector which follows symmetry.
\begin{lemma}\label{Delta set thorugh circulant matrix}
	Let $\alpha$ and $\beta$ be real numbers and 
	$\mathbf{g}=(\alpha,\beta,0,\ldots,0,\beta)^{\prime}\in \mathbb{R}^{n-1}$. 
	If 
	$G=\cir(\mathbf{g}^{\prime})$, then $(\mathbf{z}^{\prime}G)^{\prime}\in 
	\Delta$, for all 
	$\mathbf{z}\in \Delta$.
\end{lemma}
\begin{proof}
	Since $G=\cir(\mathbf{g}^{\prime})$, we have
	\begin{equation*}
		G_{*k}=\begin{cases} \alpha 
			\mathbf{e_1}+\beta(\mathbf{e_2}+\mathbf{e_{n-1}}) & 
			\text{ 
				if } k=1,\\
			\beta(\mathbf{e_{k-1}}+\mathbf{e_{k+1}})+\alpha 
			\mathbf{e_k} & \text{ if } 2\leq k\leq n-2,\\
			\beta(\mathbf{e_1}+\mathbf{e_{n-2}})+\alpha 
			\mathbf{e_{n-1}} & \text{ if } k=n-1.
		\end{cases}
	\end{equation*}
	Let $\mathbf{z}=(z_1,z_2,\ldots,z_{n-1})^{\prime}\in \Delta$. Then, 
	$z_k=z_{n+1-k}$, for 
	$k=2,3,\ldots,n-1$. Here we denote the $k$-th coordinate of 
	$\mathbf{z}^{\prime}G$ by $(\mathbf{z}^{\prime}G)_{k}$. To show 
	$(\mathbf{z}^{\prime}G)^{\prime}\in \Delta$, we 
	first consider 
	\begin{align*}
		(\mathbf{z}^\prime G)_2 =\mathbf{z}^\prime G_{*2}
		=\mathbf{z}^\prime [\beta(\mathbf{e_1}+\mathbf{e_3})+\alpha 
		\mathbf{e_2}]
		&=\beta(z_1+z_3)+\alpha z_2\\
		&=\beta(z_1+z_{n-2})+\alpha z_{n-1}\\
		&=\mathbf{z}^\prime[\beta(\mathbf{e_1}+\mathbf{e_{n-2}})
		+\alpha\mathbf{e_{n-1}}]\\	\label{second coordinate of uz}
		&=(\mathbf{z}^{\prime}G)_{n-1}.
		\intertext{If $3\leq k\leq n-2$, then $3\leq n-k+1\leq n-2$. 
			The 
			$k$-th coordinate of $\mathbf{z}^\prime G$ is given by}
		(\mathbf{z}^\prime G)_k =\mathbf{z}^\prime G_{*k}
		=\mathbf{z}^\prime 
		[\beta(\mathbf{e_{k-1}}+\mathbf{e_{k+1}})+\alpha 
		\mathbf{e_k}] 
		&=\beta(z_{k-1}+z_{k+1})+\alpha z_k\\ 
		&=\beta(z_{n+1-(k-1)}+z_{n+1-(k+1)})+\alpha z_{n+1-k}\\ 
		&=\mathbf{z}^\prime[\beta(\mathbf{e_{n+1-(k-1)}}+\mathbf{e_{n+1-(k+1)}})
		+\alpha\mathbf{e_{n+1-k}}]\\ 
		&=(\mathbf{z}^{\prime}G)_{(n+1-k)}. 
	\end{align*}
	Hence the proof.
\end{proof}

To give an alternative proof of Theorem 3 in \cite{helm 
graph}, 
let us recall 
the 
symmetric Laplacian-like matrix $\mathcal{L}$ associated with $D(H_n)^{-1}$.
\begin{definition}[\cite{helm graph}]
Let $n\geq 4$ be even. 	 For $1 \leq k \leq \frac{n}{2}-1$, 
$\beta_k=(-1)^{k}[{(n-1)-2k}]$. Let 
\begin{equation*}
\mathbf{z}=\frac{1}{2}(n+1,\beta_1, \beta_{2}, 
\ldots,\beta_{\frac{n}{2}-2},\beta_{\frac{n}{2}-1},\beta_{\frac{n}{2}-1},
\beta_{\frac{n}{2}-2}, \ldots, \beta_{2}, \beta_{1})^{\prime} \in 
\mathbb{R}^{n-1}.
\end{equation*}
.  Define 
\begin{equation}\label{Even L}
	\mathcal{L}=\begin{bmatrix}
		\frac{n-1}{2} & \frac{-1}{2}\mathbf{e}^{\prime} 
		& 
		\mathbf{0}^{\prime} \\[9pt]
		\frac{-1}{2}\mathbf{e} & 	A & 	
		-I_{n-1}\\[10pt]
		\mathbf{0}  & -I_{n-1} & 
		I_{n-1}
	\end{bmatrix}~\text{where}~ 	A=\cir(\mathbf{z^{\prime}}).
\end{equation}
\end{definition}

Now we state Theorem 3 in \cite{helm graph} and present a short proof of this.
\begin{theorem}\label{alternate proof}
	Let $n\geq 4$ be an even integer. If 		$\mathcal{L}$ is the matrix 
	defined in (\ref{Even L})
		 then 
	\begin{equation*}
		D(H_n)^{-1}=
		-\frac{1}{2}\mathcal{L}+\frac{4}{3(n-1)}\mathbf{w}\mathbf{w^{\prime}}
	\end{equation*}
where $\mathbf{w}^{\prime}=\frac{1}{4}\left(5-n,-\mathbf{e}^{\prime}, 
2\mathbf{e}^{\prime}\right)\in \mathbb{R}^{2n-1}$.
\end{theorem}
 \begin{proof}
 	Since $n$ is even, $D(H_n)$ is non-singular \cite{helm graph}. Therefore,  
 	$D(H_n)\ssymbol{2}= D(H_n)^{-1}$.
 It is straight forward that the blocks of $\mathcal{L}$ satisfy the 
 conditions $(ii)-(v)$ given in 
 Theorem \ref{necessary and sufficient conditions on L}. From 
 \cite{helm graph}, we have 
 $\sum_{k=1}^{\frac{n}{2}-1}(-1)^k(n-1-2k)=\frac{2-n}{2}$. Since 
 $A=\cir(\mathbf{z^{\prime}})$, we get
 \begin{equation*}
 	 	A\mathbf{e}= (\mathbf{z}^{\prime}\mathbf{e})\mathbf{e}=
 	 \frac{1}{2}\left( (n+1)+ 
 	 \sum_{k=1}^{\frac{n}{2}-1}2\beta_{k}\right)\mathbf{e}
 	 =\left( \frac{n+1}{2}+ 
 	 \sum_{k=1}^{\frac{n}{2}-1}(-1)^k(n-1-2k)\right)\mathbf{e}
 	 =\frac{3}{2}	\mathbf{e}.
 \end{equation*}
If we prove $(A-I)S-2I=O$ then the desired result follows by Theorem 
\ref{necessary and sufficient conditions on L} together with Remark \ref{w from 
jacklick}. By 
 (\ref{circulant properties}), it is enough to prove 
$\mathbf{z}^{\prime}S=\mathbf{s}^{\prime}+2\mathbf{e_1}^{\prime}$. We start 
with computing 
$\mathbf{z}^{\prime}S$.
Let	
$\mathbf{z}^{\prime}S=(r_1,r_2,\ldots, 
r_{n-1})$. Then 
$r_i=\mathbf{z}^{\prime}S_{*i}$. Since 
$\mathbf{z} \in 
\Delta$, $(\mathbf{z}^{\prime}S)^{\prime} \in 
\Delta$ 
by Lemma \ref{Delta set thorugh circulant matrix}. It is enough to compute the
first $\frac{n}{2}$ coordinates of 
$\mathbf{z}^{\prime}S$.
 We have 	$r_1	=\mathbf{z}^{\prime}
(2\mathbf{e_1}+\mathbf{e_2}+\mathbf{e_{n-1}})=\frac{1}{2}[2(n+1)+2\beta_1]
=(n+1)+(-1)(n-1-2)	=4$ and 
$	r_2
=\mathbf{z}^{\prime}(\mathbf{e_1}+2\mathbf{e_2}+\mathbf{e_3})
=\frac{1}{2}[(n+1)+2\beta_1+\beta_2]=1$. 
If $3 \leq i \leq \frac{n}{2}-1$, then
$r_i
=\mathbf{z}^{\prime}(\mathbf{e_{i-1}}+	2\mathbf{e_i}+\mathbf{e_{i+1}})
=\beta_{i-2}+2\beta_{i-1}+\beta_{i}=0$.
Similarly, we see that 
$r_{\frac{n}{2}}=\beta_{\frac{n}{2}-2}+3\beta_{\frac{n}{2}-1}=0$.
 Thus,
$\mathbf{z}^{\prime}S= (4,1,0,0,\ldots,0,1)= 
\mathbf{s}^{\prime}+2\mathbf{e_1}^{\prime}$. 
 \end{proof}


\subsection{The Moore-Penrose Inverse of $D(H_n)$}\label{Moore-Penrose inverse}
In this section, we derive a formula for the Moore-Penrose Inverse of $D(H_n)$,
when $n$ is odd. This is an analogous to the formula given for $D(H_n)^{-1}$. 
To obtain the desired formula, we introduce the Laplacian-like matrix $L$, 
similar to $\mathcal{L}$ in (\ref{Even L}), involving two circulant matrices 
$A$ and $B$ which are defined by the vectors  $\mathbf{x}$ and 
$\mathbf{y}$ in  $\mathbb{R}^{n-1}$ respectively. The  vectors  $\mathbf{x}$ 
and 
$\mathbf{y}$  are identified from the 
numerical examples. Hereafter it is 
assumed that $n\geq 5$  is an odd integer and 
$m=\frac{n-1}{2}$ is fixed. 

Let $1\leq k \leq m$. For each $k$,
we define $\alpha_k \in \mathbb{R}$ and we fix the vectors $\mathbf{x}$ and 
$\mathbf{y}$ in 
$\mathbb{R}^{n-1}$   which are given by
\begin{align}\label{c_k and alpha_k}
	\alpha_k&:=(-1)^{k+1}[{2m^2-6(m-k)^2+7}],\\
\label{x and y}
	\mathbf{x}&:=\frac{1}{6(n-1)}(n^2+4n-12,\alpha_1, \alpha_{2}, 
	\ldots,\alpha_{m-1},\alpha_m,
	\alpha_{m-1},\alpha_{m-2}, \ldots, \alpha_{2}, \alpha_{1})^{\prime},
	\intertext{and}	\label{x and y1}
	\mathbf{y}&:=\frac{1}{(n-1)}(2-n,-1,1,-1,1,\ldots,-1, 1,-1)^{\prime}.
\end{align}
Using Theorem \ref{necessary and sufficient conditions on L} and the above 
defined vectors $\mathbf{x}$ and 
$\mathbf{y}$, we now construct the symmetric Laplacian-like matrix $L$ 
associated with the formula for $D(H_n)\ssymbol{2}$.
\begin{definition}\label{defn of L}
	Let $n\geq 5$ be an odd integer. 
	Let
	$A=\cir(\mathbf{x^{\prime}})$ and 
	$B=\cir(\mathbf{y^{\prime}})$ where $\mathbf{x}$ and 
	$\mathbf{y}$ are given in (\ref{x and y}) and (\ref{x and y1}) 
	respectively. Define 

	\begin{equation}\label{L defn}
		L=\begin{bmatrix}
			\frac{n-1}{2} & \frac{-1}{2}\mathbf{e}^{\prime} 
			& 
			\mathbf{0}^{\prime} \\[9pt]
			\frac{-1}{2}\mathbf{e} & 	A & 	B\\[10pt]
			\mathbf{0}  & B & 
			I_{n-1}
		\end{bmatrix}.
	\end{equation}
\end{definition}

\begin{remark}
	Note that the vectors 
	$\mathbf{x}$ and 
	$\mathbf{y}$ are in $\Delta$. By  Remark  \ref{symmetry of 
	circulant 		matrices}, $A$ and $B$ are symmetric matrices of order 
	$n-1$. 
	Thus the matrix $L$, of order $2n-1$, given in the above 
	definition is symmetric.
\end{remark}
We now state the result  which gives the desired formula for 
$D(H_n)\ssymbol{2}$.
\begin{theorem}\label{Inverse Formula}
	Let $n\geq 5$ be an odd integer and $L$ be the matrix given in Definition 
	\ref{defn of L}. 
	Suppose 
	$\mathbf{w}=\frac{1}{4}\left(5-n,-\mathbf{e}^{\prime}, 
	2\mathbf{e}^{\prime}\right)^{\prime}\in \mathbb{R}^{2n-1}$ and $D$ is the 
	distance 
	matrix  
	of $H_n$. Then 
	\begin{equation*}
		D\ssymbol{2}=
		-\frac{1}{2}{L}+\frac{4}{3(n-1)}\mathbf{w}\mathbf{w^{\prime}}.
	\end{equation*}
\end{theorem}
A proof of this result is based on the following three lemmas.
\begin{lemma}\label{Ae and Be lemma}
	Let $A$ and $B$ be the matrices given in Definition 
	\ref{defn of L}. 
	Then 
		\begin{itemize}
			\item[(i)]$A\mathbf{e}
			=\frac{3}{2}\mathbf{e}$  
			\item [(ii)] $B\mathbf{e}=-\mathbf{e}$.
		\end{itemize}
\end{lemma}
\begin{proof}
	Since $A=\cir(\mathbf{x}^{\prime})$, all the row sums of 
	$A$ are 
	equal to $\mathbf{e}^{\prime}\mathbf{x}$. Therefore, 
	$A\mathbf{e}=(\mathbf{e}^{\prime}\mathbf{x})\mathbf{e}$. 
	To determine $\mathbf{e}^{\prime}\mathbf{x}$, we consider
	\begin{align}
		\notag
		\sum_{k=1}^{m-1}\alpha_k
		&=\sum_{k=1}^{m-1}(-1)^{k+1}
		(2m^2-6(m-k)^2+7)\\\label{sum alpha k upto m-1}
		&=(7-4m^2)\sum_{k=1}^{m-1}(-1)^{k+1}
		+12m\sum_{k=1}^{m-1}(-1)^{k+1}k-6\sum_{k=1}^{m-1}(-1)^{k+1}k^2.
	\end{align}	
	By a simple verification, it is easy to see that
	\begin{align}\label{sum of 1 and -1}
		\sum_{k=1}^{m-1}(-1)^{k+1}&=\begin{cases}
			1 & \text{ if $m$ is even},\\
			0 & \text{ if $m$ is odd},
		\end{cases}\\\label{sum of k}
		\sum_{k=1}^{m-1}(-1)^{k+1}k&=\frac{1}{2}\begin{cases}
			m & \text{ if $m$ is even},\\
			1-m & \text{ if $m$ is odd}, 
		\end{cases}
	\intertext{and}\label{sum of k square}
		\sum_{k=1}^{m-1}(-1)^{k+1}k^2&=\frac{1}{2}\begin{cases}
			m(m-1) & \text{ if $m$ is even},\\
			-m(m-1) & \text{ if $m$ is odd}. 
		\end{cases}
	\end{align}
	Using (\ref{sum alpha k upto m-1})-
	(\ref{sum of k 		square}) and
 $\alpha_m=(-1)^{m+1}(2m^2+7)$, 
	we 
	get
	\begin{align*}
	2\sum_{k=1}^{m-1}\alpha_k+\alpha_m		 
		&=\begin{cases} 2(-m^2+3m+7)-(2m^2+7) & 
		\text{ if $m$ is even},\\
			2(-3m^2+3m)+(2m^2+7) & \text{if $m$ is odd}.\end{cases}\\
		&=-4m^2+6m+7\\
		&=\left[-4\left(\frac{n-1}{2}\right)^2
		+6\left(\frac{n-1}{2}\right)+7\right]\\
		&=-n^2+5n+3.
	\end{align*} 
Note that  
$\mathbf{e}^{\prime}\mathbf{x}=\frac{1}{6(n-1)}[(n^2+4n-12)+	
2\sum_{k=1}^{m-1}\alpha_k+\alpha_m]$. Thus, 
$\mathbf{e}^{\prime}\mathbf{x}=\frac{3}{2}$ which implies $A\mathbf{e}
=\frac{3}{2}\mathbf{e}$. This proves (i). Since $n-1$ is 
even, we have $\mathbf{e}^{\prime}\mathbf{y}=-1$ and hence 
$B\mathbf{e}=-\mathbf{e}$.
\end{proof}
A square matrix $M$ is said to be \textit{diagonalizable} if there exists a 
non-singular matrix $P$ such that $P^{-1}MP=
\diag(\lambda_1,\lambda_{2},\ldots,\lambda_{n})$ where 
$\lambda_1,\lambda_{2},\ldots,\lambda_{n}$ are the eigenvalues of $M$. 
We now recall a result on simultaneous diagonalization which will be frequently 
used the proofs.
\begin{theorem}(\cite{Horn})\label{simultaneous diadonalization}
	Let $M_1,M_2, \ldots, M_k$  
	be diagonalizable matrices. Then $M_1,M_2, \ldots, M_k$ are simultaneously 
	diagonalizable if and only if $M_iM_j=M_jM_i$ for all $i,j\in 
	\{1,2,\ldots,k\}$. Moreover, if $\lambda_1,\lambda_{2},\ldots,\lambda_{n}$ 
	are the eigenvalues of $M_1$, then there exists a non-singular matrix $P$ 
	such that $P^{-1}M_1P=
	\diag(\lambda_1,\lambda_{2},\ldots,\lambda_{n})$ and $P^{-1}M_jP$ is 
	diagonal for all $j=2,3,\ldots,k$.
\end{theorem}
 If $\lambda_1, \lambda_2, \ldots, \lambda_k$ are the distinct eigenvalues a 
 real symmetric matrix $M$
 with respective multiplicities $m_1,m_2,\ldots,m_k$, then the set of all 
 eigenvalues of $M$ is given by 
 $\sigma(M)=\left\{\lambda_1^{[m_1]},\lambda_2^{[m_2]},\ldots,\lambda_k^{[m_k]}\right\}$.
\begin{lemma}\label{eigenvalues of B}
Let $n\geq 5$ be an odd integer. Let $B=\cir(\mathbf{y^{\prime}})$ where the 
vector $\mathbf{y}$ is given in (\ref{x and y1}). Then
$\sigma(B)=\left \{0^{[1]}, -1^{[n-2]} \right \}~ \text{and} ~ 
\sigma\left(B-\frac{1}{2(n-1)}J_{n-1}\right)=\left \{0^{[1]}, 
-1^{[n-3]},-\frac{3}{2}^{[1]}\right\}.$
\end{lemma}
\begin{proof}
Let $\mathbf{v}=(1,-1,1,-1,\ldots,1,-1)^{\prime} \in \mathbb{R}^{n-1}$. As rows 
of $\cir(\mathbf{v^{\prime}})$ are either $\mathbf{v}^{\prime}$ or 
$-\mathbf{v}^{\prime}$, we have $\ra(\cir(\mathbf{v^{\prime}}))=1$. Hence $0$ 
is 
an eigenvalue of  
$\cir(\mathbf{v^{\prime}})$ with multiplicity $n-2$. Also, 
$\cir(\mathbf{v^{\prime}})\mathbf{v}=(n-1)\mathbf{v}$. Thus, 
$\sigma(\frac{1}{n-1}\cir(\mathbf{v^{\prime}}))=\left \{1^{[1]}, 0^{[n-2]} 
\right \}$. Since 
$\mathbf{y}=\frac{1}{n-1}\mathbf{v}-\mathbf{e_1}$, we write 
 $B=\cir(\mathbf{y^{\prime}})=\frac{1}{n-1}\cir(\mathbf{v^{\prime}})-I$. 
Hence $\sigma(B)=\left \{0^{[1]}, -1^{[n-2]} \right \}$. 
It is easy to verify that
$B\mathbf{v}=\mathbf{0}$ and 	
$J_{n-1}\mathbf{v}=(\mathbf{e^{\prime}}\mathbf{v})\mathbf{e}=\mathbf{0}$. This 
gives   $B-\frac{1}{2(n-1)}J_{n-1}$ is singular because 
$\left(B-\frac{1}{2(n-1)}J_{n-1}\right)\mathbf{v}=\mathbf{0}$.
Note that $B$ and 
$\frac{1}{2(n-1)}J_{n-1}$ commute by (\ref{circulant 
	properties}).  
Hence, by Theorem 
\ref{simultaneous diadonalization},  there exists a non-singular matrix $P$  
such that 
$P^{-1}BP=\diag(0,-1,-1,\ldots,-1)$ and 
$P^{-1}\left(\frac{1}{2(n-1)}J_{n-1}\right)P=
\diag(\lambda_1,\lambda_{2},\ldots,\lambda_{n-1})$, where 
$\lambda_1,\lambda_{2},\ldots,\lambda_{n-1}$ are the eigenvalues of 
$\frac{1}{2(n-1)}J_{n-1}$. Therefore,
$P^{-1}\left(B-\frac{1}{2(n-1)}J_{n-1}\right)P=
\diag(-\lambda_1,-1-\lambda_{2},\ldots,-1-\lambda_{n-1})$. We claim that 
$\lambda_1 = 0$. On the contrary, $\lambda_1 \neq 0$. Then 
$\lambda_1=\frac{1}{2}$ and $\lambda_j=0$ for  $2\leq j \leq n-1$ because 
the eigenvalues of
$\frac{1}{2(n-1)}J_{n-1}$ are 
$\frac{1}{2}$ and $0$ with multiplicities $1$ and $n-2$ respectively.  This 
implies that $\lambda_1 \neq 0$ and   $-1-\lambda_j \neq 0$ for all 
$j=2,3,\ldots,n-1$ and hence $B-\frac{1}{2(n-1)}J_{n-1}$ is non-singular, which 
is a contradiction.
Therefore, $\lambda_1=0$, the result follows.
\end{proof}

  \begin{lemma}\label{suff lemma 2}
  	Let $n\geq 5$ be an odd integer. Suppose 
  	$S=\cir(\mathbf{s}^{\prime})$  where 	
  	$\mathbf{s}=(2,1,0,\ldots,0,1)^{\prime}$ in
  	$\mathbb{R}^{n-1}$. If $A$ and $B$ are the matrices given in the 
  	Definition 
  	\ref{defn of L}, then  the following conditions hold.
  	\begin{multicols}{2}
  		\begin{itemize}
  			\item[(i)]$A(B+I)=O$
  			\item [(ii)] 	$B(B+I)=O$
  		\item [(iii)] $BS=-S$
  		\item [(iv)] 
$(A+B)S+2B=O$.
	\end{itemize}
\end{multicols}
\end{lemma}
\begin{proof}
Since $A$ and $B$ are circulant symmetric matrices, they commute by 
(\ref{circulant properties}) and are diagonalizable. Therefore, by Theorem 
\ref{simultaneous diadonalization}, they are simultaneously diagonalizable. 
That is, there exists an invertible matrix $P$ (whose columns are the  
eigenvectors of $A$ as well $B$) such that 
$P^{-1}AP=\diag(\mu_{1},\mu_{2},\ldots,\mu_{n-1})~ \text{and}~
P^{-1}BP=\diag(\beta_{1},\beta_{2},\ldots,\beta_{n-1}).$
 By Lemma \ref{eigenvalues of B}, 
$\beta_{1}=0$ and $\beta_{i}=-1$ for 
$2\leq i \leq n-1$.
\begin{itemize}
\item [(i)] To show $A(B+I)=O$, it is enough to prove that all the eigenvalues 
of $A(B+I)$ are zero. Note that the eigenvalues of $A(B+I)$ are 
$\mu_{i}(\beta_{i}+1)$, $i=1,2,\ldots,n-1$. We claim that 
$\mu_{i}(\beta_{i}+1)=0$ for all $i$. As 
$\beta_{1}=0$ and $\beta_{i}=-1$ for 
$i=2,3,\ldots, n-1$, it remains to prove $\mu_{1}=0$. Since $B 
\mathbf{v}=\mathbf{0}=\beta_{1}\mathbf{v}$ and the multiplicity of the 
eigenvalue $0$ of $B$ is $1$, the first column of $P$ is a scalar multiple of 
$\mathbf{v}$. Therefore,  
$A\mathbf{v}=\mu_{1}\mathbf{v}$. It is easy to see that 
$A\mathbf{v}=(\mathbf{x}^{\prime}\mathbf{v})\mathbf{v}$. Consider
	$\mathbf{x}^{\prime}\mathbf{v}
	=\frac{1}{6(n-1)}\left[(n^2+4n-12)+2 \left(
	\sum_{k=1}^{m-1}(-1)^k \alpha_k\right)
	+(-1)^{m}\alpha_{m}\right]$.
{Substituting $\alpha_k$'s and simplifying, we get}
%
\begin{align*}
	\mathbf{x}^{\prime}\mathbf{v}	
	&=\frac{1}{6(n-1)}\left[(n^2+4n-12)-2\bigg((7-4m^2)\sum_{k=1}^{m-1}1
	+12m\sum_{k=1}^{m-1}k-6\sum_{k=1}^{m-1}k^2\bigg)-(2m^2+7)\right].
\end{align*}
Using the formulae for the sum of first $m-1$ natural numbers and the sum 
of 
squares 
of 
first $m-1$ natural numbers, we have
	$\mathbf{x}^{\prime}\mathbf{v}
	=\frac{1}{6(n-1)}\left[(n^2+4n-12)-2(m^2+6m-7)-(2m^2+7)\right]=0$. Thus 
	$\mu_{1}=0$.
	\item[(ii)] It is clear that all the eigenvalues of $B(B+I)$ are zero and 
	hence $B(B+I)=O$.
	\item[(iii)] 	Let	$\mathbf{v}=(1,-1,1,-1,\ldots,1,-1)^{\prime} \in 
	\mathbb{R}^{n-1}$.
	It is easy to see that   
	$\mathbf{v}^{\prime}S=\mathbf{0}^{\prime}$. Note that 
	$\mathbf{y}=\frac{1}{n-1}\mathbf{v}-\mathbf{e_1}$.
	Since 
	$B=\cir(\mathbf{y}^{\prime})$, we have $BS= 
	\cir(\mathbf{y}^{\prime}S)= 
	\cir(\frac{1}{n-1}\mathbf{v}^{\prime}S-\mathbf{e_1}S)
	=\cir(-\mathbf{s^{\prime}})=-S$.
	\item[(iv)]
	Let	$\mathbf{x}^{\prime}S=(p_1,p_2,\ldots, p_{n-1})$. 
	Then 
	$p_i=\mathbf{x}^{\prime}S_{*i}$. From (\ref{c_k and alpha_k}) and (\ref{x 
	and y}), we have 
		\begin{align*}\label{p1}\notag
			p_1	=\mathbf{x}^{\prime}
			(2\mathbf{e_1}+\mathbf{e_2}+\mathbf{e_{n-1}})&=
			\frac{1}{6(n-1)}[2(n^2+4n-12)+2\alpha_1]\\
&=\frac{2}{6(n-1)}[(n^2+4n-12)+(-1)^2(2m^2-6(m-1)^2+7)]\\
&=\frac{4n-6}{n-1}.
	\end{align*}
	Note that
	$	p_2
	=\mathbf{x}^{\prime}(\mathbf{e_1}+2\mathbf{e_2}+\mathbf{e_3})
	=\frac{1}{6(n-1)}[(n^2+4n-12)+2\alpha_1+\alpha_2]
		=\frac{n+1}{n-1}$.
	Let $3 \leq i \leq m$. Then
		$p_i
		=\mathbf{x}^{\prime}(\mathbf{e_{i-1}}+	2\mathbf{e_i}+\mathbf{e_{i+1}})
			=\alpha_{i-2}+2\alpha_{i-1}+\alpha_{i}=(-1)^{i}\frac{2}{n-1}$.
	Similarly, we see that
	$p_{m+1}= 2(\alpha_{m-1}+\alpha_{m})=(-1)^{m+1}\frac{2}{n-1}$.
	%
	%
Since $\mathbf{x} \in 
\Delta$, we have 	$(\mathbf{x}^{\prime}S)^{\prime} \in 
\Delta$ 
	by Lemma \ref{Delta set thorugh circulant matrix}. Thus,
	\begin{equation*}
	\mathbf{x}^{\prime}S=
	\frac{1}{(n-1)}(4n-6,n+1,-2,2,-2,2,\ldots,2,-2,n+1).
	\end{equation*}
	Now it is clear that $	\mathbf{x}^{\prime}S+2\mathbf{y}^{\prime}=	
	\mathbf{s}^{\prime}$. Also, from (iii), 
	$\mathbf{y}^{\prime}S=-\mathbf{s}^{\prime}$. 
	Hence 
	$(A+B)S+2B=
	\cir(\mathbf{x}^{\prime}S+\mathbf{y}^{\prime}S+2\mathbf{y}^{\prime})
	=O$.
\end{itemize}
\end{proof}
\begin{proof}[Proof of Theorem \ref{Inverse Formula}]
The result follows from Theorem \ref{necessary and sufficient conditions on L} 
and Lemmas \ref{Ae and Be lemma} and \ref{suff lemma 2}.
\end{proof}

In the following, we study two properties of the Laplacian-like matrix $L$ 
defined in (\ref{L defn}). First, we show that $L$ is a positive semidefinite 
matrix. Let us recall that an $n \times n$ real 
symmetric matrix matrix $M$ is said to 
be \textit{positive semidefinite (positive definite)} if 
$\mathbf{z^{\prime}}M\mathbf{z}\geq 0$ (respectively, 
$\mathbf{z^{\prime}}M\mathbf{z}> 0$ ) for all non-zero $\mathbf{z} \in 
\mathbb{R}^n$. We 
abbreviate the positive semidefinite (positive definite) as psd (respectively, 
pd). 

 To prove $L$ is psd, 
we need to find the eigenvalues of 
$\cir(\mathbf{s^{\prime}})$, where $\mathbf{s}=(2,1,0,0,\ldots,0,1)^{\prime}$, 
which follows from the next theorem.
\begin{theorem}(\cite{Zhang})\label{circulant eigenvalues}
Let $\mathbf{m}=(m_0,m_1,\ldots, m_{n-1})^{\prime} \in \mathbb{R}^n$ and 
$f(x)=m_0+m_1x+m_2x^2+\cdots+m_{n-1}x^{n-1}$. Then the eigenvalues of 
$M=\cir(\mathbf{m^{\prime}})$ are $f(\omega^j)$, $j=0,1,2,\ldots,n-1$ where 
$\omega=\cos(\frac{2\pi}{n})+i\sin(\frac{2\pi}{n})$ is an $n^{\text{th}}$ 
primitive root of unity.
\end{theorem}
\begin{remark}\label{eigenvalues of S}
	Consider the matrix $S=\cir(\mathbf{s^{\prime}})$ where 	
	$\mathbf{s}=(2,1,0,\ldots,0,1)^{\prime} \in
	\mathbb{R}^{n-1}$. In this case, $f(x)=2+x+x^{n-2}$. Therefore, 
	$f(\omega_j)=2+\omega_j+\omega_j^{-1}$ because $\omega_j^{n-1}=1$. From 
	Theorem \ref{circulant 
		eigenvalues}, the eigenvalues of $S$ are $4\cos^2(\frac{\pi j}{n-1})$, 
		$j=0,1,2,\ldots,n-2$.
\end{remark}
Using the same arguments as given in the proof of Theorem $7.2 (iv)$ in 
\cite{Balaji 
	odd wheel graph}, it can be shown that $L$ is psd. However, we wish to 
give a different proof based on the following result involving the Schur 
complement.
\begin{theorem}(\cite{Horn})\label{psd equiv cond for block marix}
	Let $M_1$ and $M_3$ be square matrices.
	If $M=\begin{bmatrix}
	M_1 & M_2\\
	M_2^{\prime}  & M_3
\end{bmatrix}$ is a symmetric matrix then the following statements are true.
\begin{itemize}
\item[(i)] Let $M_1$ be pd. Then
$M$ is psd if and only if $M_3-M_2^{\prime}M_1^{-1}M_2$ is psd.
\item[(ii)]  Let $M_3$ be pd. Then $M$ is psd if and only if 
$M_1-M_2M_3^{-1}M_2^{\prime}$ is psd.
\end{itemize}
\end{theorem}

\begin{theorem}
	Let $L$ be the matrix given in  Definition \ref{defn of L}. Then $L$ is a 
	positive 
	semidefinite matrix.
\end{theorem}
\begin{proof}
By $(i)$ of Theorem \ref{psd equiv cond for 
	block 
	marix}, it suffices to show that
$X=\left[\begin{smallmatrix}
		A-\frac{1}{2(n-1)}J_{n-1} &&& 	B\\[3pt]
	 B &&& 	I_{n-1}
\end{smallmatrix}\right]$ is psd. Note that $B^2=-B$ by $(ii)$ of Lemma 
\ref{suff lemma 
2}. Now applying part $(ii)$ of Theorem \ref{psd equiv cond for block 
marix}  to $X$, we get $L$ is psd if and only if 
$A+B-\frac{1}{2(n-1)}J_{n-1}$ is psd. If we prove all the eigenvalues of 
$A+B-\frac{1}{2(n-1)}J_{n-1}$ are non-negative, then the desired result follows.
Since $A$, $B$, $S$ and $B-\frac{1}{2(n-1)}J_{n-1}$ are symmetric circulant  
matrices, by (\ref{circulant properties}) and Theorem 
\ref{simultaneous diadonalization},  there exists an invertible matrix $P$  
such that 
$P^{-1}AP=:\Lambda_1=\diag(\mu_{1},\mu_{2},\ldots,\mu_{n-1})$,  
$P^{-1}BP=:\Lambda_2=\diag(0,-1,-1,\ldots,-1)$,
$P^{-1}SP=:\Lambda_3=\diag(\gamma_{1},\gamma_{2},\ldots,\gamma_{n-1})$,  and 
$P^{-1}\left(B-\frac{1}{2(n-1)}J_{n-1}\right)P
=\diag(\delta_{1},\delta_{2},\ldots,\delta_{n-1})$ where the diagonal entries 
of 
$\Lambda_2$ are obtained from Lemma \ref{eigenvalues of B}.
Then $\mu_{1}=0$ follows from 
$\Lambda_1(\Lambda_2+I)=O$ by $(i)$ of Lemma \ref{suff lemma 2}. We have 
$(\Lambda_2+I)\Lambda_3=O$ by the identity $BS=-S$ from Lemma \ref{suff lemma 
2}. This implies $\gamma_{1}=0$.  We now claim that 
$\delta_{1}=0$. The multiplicities of the eigenvalue $0$ with respect to $B$ 
and $B-\frac{1}{2(n-1)}J_{n-1}$ are equal to $1$, see Lemma \ref{eigenvalues of 
B}. Also, 
$B\mathbf{v}=\left(B-\frac{1}{2(n-1)}J_{n-1}\right)\mathbf{v}=\mathbf{0}$ where 
$\mathbf{v}=(1,-1,1,-1,\ldots,1,-1)^{\prime}\in \mathbb{R}^{n-1}$. Therefore, 
the first column of $P$ must be a scalar multiple of $\mathbf{v}$. Hence 
$\delta_{1}=0$. To complete the proof, we need to show that $\mu_j+\delta_{j} 
\geq 
0$ for all $j=2,3,\ldots,n-1$.  Let 
$2 \leq j \leq n-1$. From Lemma \ref{eigenvalues of B}, it is clear that 
$\delta_{j}$ is either $-1$ or $-\frac{3}{2}$. In view of this, it is enough to 
prove that $\mu_{j}\geq \frac{3}{2}$. Using the fact that $\ra(S)=n-2$ (see 
Theorem \ref{rank of H}) 
and by Remark 
\ref{eigenvalues of S}, we get $\gamma_{j}>0$.  By 
$(iii)$ and $(iv)$ of Lemma \ref{suff lemma 2}, we write $(A-I)S=-2B$ which 
gives $(\Lambda_1-I)\Lambda_3=-2\Lambda_2$. This yields that 
$(\mu_{j}-1)\gamma_{j}=2$. This implies 
$\mu_j=\frac{2}{\gamma_j}+1=\frac{2}{4\cos^2\left(\frac{\pi 
j}{n-1}\right)}+1\geq \frac{3}{2}$.
This completes the proof.
\end{proof}
In the next result, we find the rank of $L$. We will make use of the following 
theorem.
\begin{theorem}(\cite{Horn})\label{rank inequality}
	Let $A$ and $B$ be symmetric matrices of same order. Then $\ra(A+B) \leq 
	\ra(A)+\ra(B)$. Furthermore, equality holds if and only if $R(A)\bigcap 
	R(B)=\{\mathbf{0}\}$.
\end{theorem}
\begin{theorem}
	Let $L$ be the matrix given in $(\ref{L defn})$. Then the rank of $L$ 
	is $2n-3$.
\end{theorem}
\begin{proof}
	For 	$\mathbf{w}=\frac{1}{4}\left(5-n,-\mathbf{e}^{\prime}, 
	2\mathbf{e}^{\prime}\right)^{\prime}\in \mathbb{R}^{2n-1}$, we have 
	$\mathbf{e}^{\prime}\mathbf{w}=1$. We claim that $\mathbf{w} \notin R(L)$. 
	On the contrary, suppose $L\mathbf{z}=\mathbf{w}$ for some $\mathbf{z} \in 
	\mathbb{R}^{2n-1}$. Since $\mathbf{e}^{\prime}L=\mathbf{0}^{\prime}$, 
	 $\mathbf{e}^{\prime}L\mathbf{z}=\mathbf{e}^{\prime}\mathbf{w}=0$ which is 
	 impossible. Hence $\mathbf{w} \notin R(L)$. Thus $R(L)\bigcap 
	R(\mathbf{w}\mathbf{w}^{\prime})=\{\mathbf{0}\}$. By Theorems \ref{Inverse 
	Formula} and \ref{rank inequality}, $\ra(L)=\ra(D(H_n)\ssymbol{2})-1$. 
	Since $\ra(D(H_n))=\ra(D(H_n)\ssymbol{2})$, the result follows from Theorem 
	\ref{rank of H}.
\end{proof}

\end{document}